\documentclass[11pt, onesided, leqno]{amsart}
\usepackage{enumerate,enumitem}
\usepackage[all]{xy}
\usepackage{mathtools}
\usepackage{amsmath,amssymb,mathrsfs,geometry,color,bm}
\usepackage{pb-diagram}
\usepackage{bbm}

\newcommand{\doublespace}
   {\addtolength{\baselineskip}{0.25\baselineskip}}

\newtheorem{thm}{Theorem}[section]

\newtheorem{lem}[thm]{Lemma}

\theoremstyle{definition}

\newtheorem{remark}[thm]{Remark}
\newtheorem{example}[thm]{Example}

\numberwithin{equation}{section}

\newcommand{\R}{\mathbb{R}}   
\newcommand{\C}{\mathbb{C}}   
\newcommand{\N}{\mathbb{N}}   
\renewcommand{\epsilon}{\varepsilon}    

\newcommand{\calM}{{\mathcal{M}}} 
 
\newcommand{\calP}{{\mathcal{P}(\mathcal{M})}} 
\newcommand{\calT}{{\mathscr{J}}}

\newcommand{\calF}{{\Psi}}
\newcommand{\calG}{{\Phi}}
\newcommand{\scrF}{{\mathscr{F}}}
\newcommand{\bbm}{{\mathbbm{I}}}

\title[Rates of convergence for laws of the spectral maximum of free random variables]{Rates of convergence for laws of the spectral maximum of free random variables}

\author{Yuki Ueda}


\address{
Yuki Ueda: 
Department of Mathematics, Hokkaido University of Education, 9 Hokumon-cho, Asahikawa, Hokkaido 002-8501, Japan}
\email{ueda.yuki@a.hokkyodai.ac.jp}

\begin{document}

\maketitle  
\doublespace
\pagestyle{myheadings} 
%

\begin{abstract}
Let $\{X_n\}_n$ be a sequence of freely independent, identically distributed non-commutative random variables. Consider a sequence $\{W_n\}_n$ of the renormalized spectral maximum of random variables $X_1,\cdots, X_n$. It is known that the renormalized spectral maximum $W_n$ converges to the free extreme value distribution under certain conditions on the distribution function. In this paper, we provide a rate of convergence in the Kolmogorov distance between a distribution function of $W_n$ and the free extreme value distribution. 
\end{abstract}



\section{Introduction}

\subsection{Classical extreme value theory}
Let $\{X_i\}_i$ be a sequence of independent, identically distributed random variables on a probability space $(\Omega, \mathcal{F},\mathbb{P})$ and $G(\cdot)=\mathbb{P}(X_i\le \cdot)$ for all $i\in \N$. If there are $a_n>0$ and $b_n\in \R$ such that a random variable
$$
W_n:=\frac{\max\{X_1,\cdots, X_n\}-b_n}{a_n}
$$
converges to some random variable in distribution, then its distribution function is characterized by one of the following:
\begin{align*}
\calG_0(x):&=\exp(-e^{-x})\bbm_\R(x), \qquad \gamma=0 \hspace{2mm}\text{(Gumbel distribution)};\\
\calG_\gamma(x):&=\exp(-x^{-\gamma})\bbm_{(0,\infty)}(x), \hspace{2mm} \gamma>0 \hspace{2mm} \text{(Fr\'{e}chet distribution)};\\
\calG_\gamma(x):&=\exp(-|x|^{-\gamma})\bbm_{(-\infty,0)}(x)+\bbm_{[0,\infty)}(x), \hspace{2mm} \gamma<0 \hspace{2mm}\text{(Weibull distribution)},
\end{align*}
where $\bbm_B$ is the indicator function of a subset $B$ of $\R$. The distribution function $\Phi_\gamma$ is called the {\it extreme value distribution}. In this case, $G$ is said to be {\it in the max-domain of attraction of }$\Phi_\gamma$, denoted by $G\in \mathcal{D}_\ast(\Phi_\gamma;a_n,b_n)$ (see \cite{G43}, \cite{H70}, \cite{R87} for details). Note that a distribution function of $W_n$ is given by $x\mapsto G^n(a_nx+b_n)$ for all $x\in \R$. Assume that the distribution function of $W_n$ has a differentiable density $g_n$ for each $n\in\N$, i.e. $g_n(x)=\frac{d}{dx} G^n(a_nx+b_n)=na_nG^{n-1}(a_nx+b_n)g_n(a_nx+b_n)$ is differentiable at $x$. One of main interests in a view of probability theory and statistics is to study rates of convergence in the Kolmogorov distance which is the maximum distance between distribution functions on $\R$, that is, the Kolmogorov distance $d_K(F_1,F_2)$ between distribution functions $F_1$ and $F_2$ on $\R$ is defined by
$$
d_K(F_1,F_2):=\sup_{x\in \R} |F_1(x)-F_2(x)|.
$$ 
If $G\in \mathcal{D}_\ast(\Phi_\gamma;a_n,b_n)$, then we obtained
\vspace{-1.5mm}
\begin{align*}
d_K( G^n(a_n\cdot+b_n), \Phi_\gamma) \le C_\gamma \mathbb{E}[|\Gamma_{\ast,\gamma,n}(W_n)|],
\end{align*}
where $C_\gamma>0$ is the constant depending only on $\gamma$, and
\begin{align*}
\Gamma_{\ast,\gamma,n}(x):=\begin{cases}
1-e^x\{1+(\log g_n(x))'\}, &  \gamma=0,\\
|\gamma|-\text{sgn}(\gamma)(\gamma+1)|x|^\gamma-|x|^{\gamma+1}(\log g_n(x))', &  \gamma\neq0,
\end{cases}
\end{align*}
where 
$$
\text{sgn}(\gamma):=\begin{cases}
1, &\gamma>0,\\
-1, & \gamma<0.
\end{cases}
$$
The result on Fr\'{e}chet case ($\gamma>0$) was answered by Smith \cite{Smith} and Omey \cite{Omey}. After that, Bartholm\'{e} and Swan \cite{BS} also studied the same situation in Fr\'{e}chet case via the Stein's method. Weibull case ($\gamma<0$) can be proceed almost similarly to Fr\'{e}chet case. Recently, Gumbel case ($\gamma=0$) was also investigated by Kusumoto and Takeuchi \cite{KT20} by the Stein's method.

\subsection{Free extreme value theory}
In non-commutative probability theory, self-adjoint operators are interpreted as real-valued random variables. A remarkable phenomenon in non-commutative probability theory is an appearance of several kind of independence of random variables. In particular, free independence is the most important independence to understand various fields. Free probability theory is a non-commutative probability theory considering free random variables and it was initiated by Voiculescu to understand the free product of von Neumann algebras. In 1991, Voiculescu found a remarkable relation between free probability theory and random matrix theory (see \cite{V91} for further details). Recently, free probability theory contributes to the development of quantum information theory (see e.g. \cite{CHN17}). 

In 2006, Ben Arous and Voiculescu \cite{BV06} started free extreme value theory to understand the maximum (in Ando's sense, see \cite{A89}, \cite{BV06}, \cite{O71} for details) of free random variables. We denote by $F_X\Box \hspace{-.95em}\lor F_Y$ as a distribution function of the spectral maximum $X\lor Y$ (see Section 2 for further details) of free random variables $X$ and $Y$ with distribution functions $F_X$ and $F_Y$, respectively. The symbol $\Box \hspace{-.72em}\lor $ is called {\it free max-convolution}. Write $F^{\Box\hspace{-.55em}\lor n}$ as the $n$-fold free max-convolution of $F$. Let $\{X_i\}_i$ be a family of freely independent identically distributed random variables such that $X_i$ is distributed as a distribution function $U$ on $\R$. A distribution function of the renormalized spectral maximum 
$$
W_n:=\frac{X_1\lor \cdots \lor X_n-b_n}{a_n}, \qquad a_n>0, \hspace{2mm} b_n\in \R,
$$ 
is given by $x\mapsto U^{\Box\hspace{-.55em}\lor n}(a_nx+b_n)$. If there are $a_n>0$ and $b_n\in \R$ such that $\lim_{n\rightarrow \infty} U^{\Box\hspace{-.55em}\lor n}(a_nx+b_n)=V(x)$ for all continuous points $x$ of $V$ (actually, $x\in\R$), then $V$ is called the {\it free extreme value distribution}, denoted by $\calF_\gamma$ for some $\gamma\in \R$. In this case, $U$ is said to be {\it in the free max-domain of attraction of } $\calF_\gamma$, denoted by $U\in \mathcal{D}_{\boxplus}(\Psi_\gamma;a_n,b_n)$. After \cite{BV06}, free extreme value theory kept on advancing from several views. Ben Arous and Kargin \cite{BK10} contributed to clarify relations between convergence of free point processes and the max-domains of attraction of free extreme value distributions. Moreover, Benaych-Georges and Cabanal-Duvillard \cite{BC10} constructed random matrices such that those empirical eigenvalue distributions converge to the free extreme value distribution as matrix size goes to infinity. Grela and Nowak \cite{GN17} characterized the class of distribution functions which are in the free max-domains of attraction of free extreme value distributions. As an other advance, Vargas and Voiculescu \cite{VV18} started boolean extreme value theory (the reader can find basic concepts on boolean probability theory in \cite{SW97}). Recently, the author \cite{U19} found relations between limit theorems for classical, free and boolean max-convolutions (this result is max-analogue of Bercovici-Pata's work \cite{BP99}). Further, the author \cite{U20} and the author and Hasebe \cite{HU20} also found several relations between additive convolution semigroups and max-convolution semigroups in classical, free and boolean cases via Tucci, Haagerup and M\"{o}ller's limit theorem (see \cite{T10}, \cite{HM13}).

\subsection{Main result}
Our main interest here is to obtain a new rate of convergence toward free extreme value distributions. For $n\ge2$, we assume that $W_n$ has a density $u_n$ which satisfies $\text{supp}(u_n):=\overline{\{x: u_n(x)>0\}}=[ A_n, B_n]$ and is differentiable on $( A_n, B_n)$, where $ A_n< B_n$. Furthermore, we request a few of assumptions of $u_n$. Through the paper, we assume that $ A_n>-\infty$ since the $n$-fold free max-convolution cuts the left side of the original distribution function. If $U\in \mathcal{D}_{\boxplus}(\Psi_\gamma;a_n,b_n)$ and $n\ge2$, then we obtain
\begin{equation}\label{mainresult}
d_K(U^{\Box\hspace{-.55em}\lor n}(a_n\cdot+b_n), \Psi_\gamma) \le  \int_{ A_n}^{ B_n} |\Gamma_{\boxplus,\gamma,n}(x)|u_n(x)dx +r_{ A_n} u_n( A_n+),
\end{equation}
under a decay condition of the function $u_n$ at $\beta$,
where 
\begin{align*}
\Gamma_{\boxplus,\gamma,n}(x):=\begin{cases}
1+(\log u_n(x))', &\gamma=0,\\
1+\gamma^{-1}(1+x (\log u_n(x))'), &\gamma\neq0
\end{cases}
\end{align*}
and
\begin{align*}
r_{ A_n}:=\begin{cases}
1-e^{-| A_n|}, & \gamma=0,\\
\gamma^{-1} A_n\{1-(| A_n|\land | A_n|^{-1})^{|\gamma|}\}, &  \gamma \neq 0.
\end{cases}
\end{align*}
A different point from classical case is to remain information on the boundary of $W_n$. The remainder $r_{A_n}$ is determined by $\gamma$ and $A_n$. In particular, if a sample distribution $U$ is the classical distribution $\Phi_\gamma$ for $\gamma\in \R$, then $U=\Phi_\gamma\in \mathcal{D}_{\boxplus}(\Psi_\gamma;a_n,b_n)$ and
\begin{equation*}
d_K(\Phi_\gamma^{\Box \hspace{-.55em}\lor n} (a_n \cdot+b_n),\Psi_\gamma)\le \frac{1}{n},
\end{equation*}
where $a_n>0$ and $b_n\in \R$ are suitable constants. Note that rates of convergence toward the free extreme value distributions are the same even though $\gamma$ is an arbitrary real number.

\section{Preliminary}

Let $(\mathcal{M},\tau)$ be a tracial $W^*$-probability space, that is, $\mathcal{M}$ is a von Neumann algebra and $\tau:\mathcal{M}\rightarrow\C$ is a tracial normal faithful state on $\mathcal{M}$. We may assume that $\mathcal{M}$ acts on a Hilbert space $\mathcal{H}$ by taking the non-commutative $L^2$-space $\mathcal{H}=L^2(\mathcal{M},\tau)$. Denote by $\calM_b(I)$ the set of all Borel measurable and bounded functions on a subset $I$ of $\R$. A self-adjoint operator $X$ on $\mathcal{H}$ is said to be {\it affiliated with $\calM$} if the Borel functional calculus $f(X)$ is in $\mathcal{M}$ for any $f\in \calM_b(\R)$. Note that $X$ is a bounded self-adjoint operator affiliated with $\calM$ if and only if $X\in\calM$. In this paper, we call a self-adjoint operator affiliated with $\calM$ a {\it (non-commutative) random variables}. A {\it spectral distribution function of $X$ with respect to $\tau$} is defined by 
$$
\scrF_X(x):=\tau(\bbm_{(-\infty,x]}(X)), \qquad x\in \R.
$$ 
Denote by $\calP$ the set of all projections in $\calM$ and denote also by $\calM_{sa}$ the set of all self-adjoint operators in $\calM$. For $P,Q\in\calP$, we define  $P\lor Q$ as the projection onto $(P\lor Q)\mathcal{H}:=\text{cl}(P\mathcal{H}+Q\mathcal{H})$. Then $P\lor Q\in\calP$ and it is the maximum of $P$ and $Q$ with respect to the usual operator order. However, it was shown in \cite{K51} that the maximum of two self-adjoint operators in $\mathcal{B}(\mathcal{H})$ with respect to the operator order does not necessarily exist. Instead of the operator order, Olson \cite{O71} (see also \cite{A89}) introduced the spectral order to define the maximum of (bounded) self-adjoint operators on $\mathcal{H}$. After that, the spectral order was defined in the level of general von Neumann algebras as follows. For $X,Y\in\calM_{sa}$, we define $X\prec Y$ if we have $\bbm_{(x,\infty)}(X)\le \bbm_{(x,\infty)}(Y)$ for all $x\in\mathbb{R}$. The order $\prec$ is called the {\it spectral order}. For any $X,Y\in\calM_{sa}$, we define $X\lor Y\in\calM_{sa}$ by
\begin{align*}
\bbm_{(x,\infty)}(X\lor Y):=\bbm_{(x,\infty)}(X)\lor \bbm_{(x,\infty)}(Y), \qquad x\in\mathbb{R}.
\end{align*}
The operator $X\lor Y$ is well-defined since the right-hand side in the above identity is projection-valued, decreasing and right-continuous in the strong operator topology as a function of $x$. Moreover, $X\lor Y$ is the maximum of $X$ and $Y$ with respect to the spectral order. Finally, the definition was extended the spectral order to the set of all (unbounded) selfadjoint operators affiliated with $\calM$. Note that if $X$ and $Y$ are self-adjoint operators affiliated with $\calM$, then so is $X\lor Y$ (see \cite{C06} and \cite{BV06}). Ben Arous and Voiculescu \cite{BV06} finally calculated spectral distributions of maximum $X\lor Y$ of free random variables $X$ and $Y$ affiliated with $\calM$ as follows: $\scrF_{X\lor Y}=\max\{\scrF_X+\scrF_Y-1,0\}$. According to this calculation, we define the free max-convolution $F {\Box \hspace{-.75em} \lor} G:=\max\{F+G-1,0\}$ for any distribution functions $F,G$ on $\mathbb{R}$. Denote by $F^{\Box\hspace{-.55em}\lor n}$ the $n$-fold free max-convolution of a distribution function $F$ on $\R$ and it is calculated inductively as follows: $F^{\Box\hspace{-.55em}\lor n}=\max\{nF-(n-1),0\}$ for each $n\in \N$.
A non-degenerate distribution function $V$ is said to be {\it freely extreme value} if for any $n\in\mathbb{N}$, there exist distribution functions $U$, $a_n>0$ and $b_n\in\mathbb{R}$ such that $U^{\Box \hspace{-.55em} \lor n} (a_n \cdot+b_n) \xrightarrow{w} V(\cdot)$ as $n\rightarrow\infty$, where $\xrightarrow{w}$ denotes the weak convergence. 
In \cite{BV06}, a non-degenerate distribution function $V$ is free extreme value if and only if there exist $a>0$ and $b\in\mathbb{R}$ such that $V(ax+b)$ is one of the following distributions
\begin{align*}
\calF_0(x):&=(1-e^{-x})\bbm_{[0,\infty)}(x), \hspace{2mm} \gamma=0\hspace{2mm} \text{(free Gumbel distribution)}; \\
\calF_\gamma(x):&=(1-x^{-\gamma})\bbm_{[1,\infty)}(x), \hspace{2mm} \gamma>0 \hspace{2mm}\text{(free Fr\'{e}chet distribution)}; \\
\calF_\gamma(x):&=\{1-|x|^{-\gamma}\}\bbm_{[-1,0]}(x)+\bbm_{(0,\infty)}(x), \hspace{2mm} \gamma<0 \hspace{2mm}\text{(free Weibull distribution)}.
\end{align*}

\section{Free Gumbel case}

Let $Z_0$ be a random variable such that $\scrF_{Z_0}=\calF_0$ and $\{X_i\}_i$  a family of freely independent, identically distributed random variables. Assume that $X_i$ is distributed as a distribution function $U$ on $\R$ such that $U\in \mathcal{D}_{\boxplus}(\Psi_0;a_n,b_n)$, that is, there exist $a_n>0$ and $b_n\in\R$ such that $W_n=(X_1\lor\cdots\lor X_n-b_n)/a_n$ converges to $Z_0$ {\it in distribution}, i.e., $\scrF_{W_n}\xrightarrow{w} \calF_0$ as $n\rightarrow\infty$. For each $n\ge2$, we assume that $W_n$ has a density $u_n$ with respect to the Lebesgue measure on $\R$, that is, $\tau(f(W_n))=\int_\R f(x)u_n(x)dx$ for all $f\in\calM_b(\R)$, and the function $u_n$ satisfies
\begin{align}
&\text{supp}(u_n)=[ A_n, B_n] \text{ for some } -\infty< A_n< B_n \le\infty,   \label{G-Cond1}\\
&u_n( A_n+):=\lim_{t\rightarrow  A_n+0} u_n(t)\in [0,\infty) \text{ and } u_n( B_n-): =\lim_{t\rightarrow  B_n-0} u_n(t)=0, \label{G-Cond1-1}\\
&u_n\text{ is differentiable on }( A_n, B_n), \label{G-Cond2} \\
&\rho_n:=(\log u_n)' \text{ is bounded and continuous on } ( A_n, B_n),  \label{G-Cond3}
\end{align}
where $[ A_n, B_n]=[ A_n,\infty)$ if $ B_n=\infty$. 

We define a differential operator $\calT_0$ as follows:
$$
\calT_0 \varphi(w):=\varphi'(w)-\varphi(w), 
$$
for all absolutely continuous functions $\varphi$ on $\R$. 

Suppose that $x\in \R$. Consider a differentiable equation
\begin{equation}\label{G-diff}
\calT_0\varphi(w)=\bbm_{(-\infty, x]}(w)-\Psi_0(x).
\end{equation}
Using the definition of $\calT_0$ and the product rule, we have
$$
\frac{d}{dw}\left(e^{-w} \varphi(w) \right)=e^{-w} (\bbm_{(-\infty, x]}(w)-\Psi_0(x)).
$$
Then a bounded solution of the equation \eqref{G-diff} is given by
$$
\varphi(w) =-e^w \int_w^\infty \{\bbm_{(-\infty, x]}(t)-\Psi_0(x)\}e^{-t}dt,
$$
denoted by $\varphi_x$ for each $x\in \R$. The function $\varphi_x$ is explicitly written as follows. Let $x<0$. Then
\begin{align}\label{G-varphi1}
\varphi_x(w)=\begin{cases}
e^{w-x}-1, & w\le x,\\
0, & x<w.
\end{cases}
\end{align}
Let $x>0$. Then
\begin{align}\label{G-varphi2}
\varphi_x(w)=\begin{cases}
e^{w-x}-e^{-x}, & w\le x,\\
1-e^{-x}, & x<w.
\end{cases}
\end{align}

Note that $\|\varphi_x\|_\infty \le 1$. Then we have $|\varphi_x(t)u_n(t)|\le u_n(t)\rightarrow0$ as $t\rightarrow B_n-0$ by \eqref{G-Cond1-1}. Therefore we get 
\begin{align}\label{G-beta}
\lim_{t\rightarrow  B_n-0}\varphi_x(t)u_n(t)=0.
\end{align}
For each $x\in \R$ and $ A_n>-\infty$, we define 
$$
\eta_{0, A_n,x}:=\lim_{t\rightarrow A_n+0} \varphi_x(t)u_n(t).
$$ 

For $n\ge2$, we define
$$
\calT_{0,n}\varphi(w):=\varphi'(w)+\varphi(w)\rho_n(w),
$$
for all absolutely continuous functions $\varphi$ on $\R$.

\begin{lem}\label{G-Lem1}
For each $n\ge2$, we obtain
$$
\tau(\calT_{0,n}\varphi_x(W_n))=- \eta_{0, A_n,x}.
$$
Furthermore, we get $|\eta_{0, A_n,x}|\le (1-e^{-| A_n|})u_n( A_n+)<\infty$.
\end{lem}
\begin{proof}
Since $W_n$ is a self-adjoint operator affiliated with $\calM$, we get $\calT_{0,n}\varphi_x(W_n)\in \calM$ by the assumption \eqref{G-Cond3} and the definition of $\calT_{0,n}$ and representations \eqref{G-varphi1} and \eqref{G-varphi2}. Furthermore, by definitions of $\calT_{0,n}$ and $\eta_{0, A_n,x}$ and the condition \eqref{G-beta}, we observe 
\begin{align*}
\tau(\calT_{0,n}\varphi_x(W_n))&=\int_{ A_n}^{ B_n} \calT_{0,n}\varphi_x(t) u_n(t)dt\\
&=\int_{ A_n}^{ B_n} (\varphi_x'(t)+\varphi_x(t) \rho_n(t)) u_n(t)dt\\
&=\int_{ A_n}^{ B_n}(\varphi_x'(t)u_n(t)+\varphi_x(t) u_n'(t))dt\\
&=\int_{ A_n}^{ B_n} (\varphi_x(t)u_n(t))'dt\\
&=\lim_{t\rightarrow B_n-0} \varphi_x(t)u_n(t)-\eta_{0, A_n,x}\\
&=-\eta_{0, A_n,x}.
\end{align*}
It follows from the explicit representations \eqref{G-varphi1} and \eqref{G-varphi2} that $|\varphi_x(A_n)|\le 1 - e^{A_n}$ if $A_n<0$ and $|\varphi_x(A_n) |\le 1 - e^{-A_n}$ if $A_n\ge0$. Hence $|\varphi_x(A_n)|\le 1-e^{-|A_n|}$ and therefore $|\eta_{0, A_n,x}|\le (1-e^{-| A_n|})u_n( A_n+)$. By the assumption \eqref{G-Cond1-1}, we obtain $(1-e^{-| A_n|})u_n( A_n+)<\infty$.
\end{proof}

The next result is an immediate consequence of Lemma \ref{G-Lem1}.
\begin{thm}\label{G-Thm}
For each $n\ge2$, we obtain
\begin{align*}
d_K(\scrF_{W_n}, \Psi_0)\le \tau(|1+\rho_n(W_n)|)+(1-e^{-| A_n|}) u_n( A_n+).
\end{align*}
\end{thm}
\begin{proof}
For each $x\in \R$, we have
\begin{align*}
\scrF_{W_n}(x)&- \Psi_0(x)\\
&=\tau(\bbm_{(-\infty, x]}(W_n))- \Psi_0(x)\\
&=\tau(\calT_0\varphi_x(W_n)) \qquad \text{ (by definition of $\calT_0$)}\\
&=\tau(\calT_0\varphi_x(W_n))-\tau(\calT_{0,n}\varphi_x(W_n))+(-\eta_{0, A_n,x}) \hspace{2mm} \text{(by Lemma \ref{G-Lem1})}\\
&=\tau\left(\varphi_x'(W_n)-\varphi_x(W_n)-\left(\varphi_x'(W_n)+\varphi_x(W_n)\rho_n(W_n)\right)\right)-\eta_{0, A_n,x}\\
&=-\tau(\varphi_x(W_n)(1+\rho_n(W_n)))-\eta_{0, A_n,x}.
\end{align*}
Consequently, for all $x\in \R$, we get
\begin{align*}
|\scrF_{W_n}(x)- \Psi_0(x)|&\le \tau(|\varphi_x(W_n)(1+\rho_n(W_n))|) + |\eta_{0, A_n,x}|\\
&\le \tau(|1+\rho_n(W_n)|)+|\eta_{0, A_n,x}| \qquad \text{(by $\|\varphi_x\|_\infty\le 1$)}\\
&\le \tau(|1+\rho_n(W_n)|)+ (1-e^{-| A_n|}) u_n( A_n+),
\end{align*}
by Lemma \ref{G-Lem1}. By taking the supremum of the left hand side for $x$, we get the inequality in this theorem.
\end{proof}

\begin{remark}
For $n\ge2$, we assume that the function $u_n$ is written by
$$
u_n(x)=C_n(x) e^{-x}, \qquad x\in ( A_n, B_n),
$$
where $C_n:( A_n, B_n)\rightarrow (0,\infty)$ is differentiable on $( A_n, B_n)$ and satisfies that
\begin{align}
&\sup_{n\in \N} C_n( A_n+)<\infty, \label{GCond4-1}\\
&\lim_{n\rightarrow\infty}\sup_{x\in( A_n, B_n)} \left|\frac{C_n'(x)}{C_n(x)}\right|=0 \label{GCond4-2}.
\end{align}
Moreover we assume that $ A_n\rightarrow0$ as $n\rightarrow\infty$. Then we have
\begin{align*}
d_K(\scrF_{W_n},\Psi_0)&\le \tau(|1+\rho_n(W_n)|)+ (1-e^{-| A_n|}) u_n( A_n+) \\
&=\int_{ A_n}^{ B_n} \left|\frac{C_n'(x)}{C_n(x)}\right|u_n(x)dx+e^{- A_n}(1-e^{-| A_n|})C_n( A_n+)\\
&\le \sup_{x\in ( A_n, B_n)} \left|\frac{C_n'(x)}{C_n(x)}\right| \int_{ A_n}^{ B_n} u_n(x)dx+ e^{- A_n}(1-e^{-| A_n|}) \sup_{n\in\N}C_n( A_n+)\\
&= \sup_{x\in ( A_n, B_n)} \left|\frac{C_n'(x)}{C_n(x)}\right| + e^{- A_n}(1-e^{-| A_n|}) \sup_{n\in\N}C_n( A_n+)\\
&\xrightarrow{n\rightarrow\infty} 0,
\end{align*}
by the conditions \eqref{GCond4-1} and \eqref{GCond4-2} and that $ A_n\rightarrow0$ as $n\rightarrow\infty$.
\end{remark}

\begin{example}\label{G-ex}
Suppose that $n\ge2$. Assume that $X_i$ is distributed as the Gumbel distribution $\calG_0$. Since $\calG_0\in \mathcal{D}_{\boxplus}(\Psi_0;1,\log n)$, the random variable $W_n=X_1\lor \cdots \lor X_n-\log n$ converges to $Z_0$ in distribution as $n\rightarrow\infty$. A spectral distribution function of $W_n$ is given by
$$
\scrF_{W_n}(t)=\calG_0^{\Box\hspace{-.55em}\lor n} (t+\log n)=\max\left\{n\exp\left(-\frac{e^{-t}}{n} \right)-(n-1),0\right\}, \qquad t\in\R,
$$
so that a density function $u_n$ of the random variable $W_n$ is 
$$
u_n(t):=\frac{d}{dt}\scrF_{W_n}(t)=e^{-t}\exp\left(-\frac{e^{-t}}{n}\right),
$$
for all $t> A_n:= -\log \left\{ -n\log\left(1-\frac{1}{n}\right)\right\}\in(-\infty,0)$, where $ B_n=\infty$. A direct calculation shows that
\begin{align*}
u_n( A_n+)&=-(n-1) \log \left(1-\frac{1}{n} \right)\in (0,\infty),\\
u_n( B_n-)&=0.
\end{align*}
It is easy to see that $u_n$ is differentiable on $( A_n,\infty)$. Moreover, we have
$$
\rho_n(t):=(\log u_n(t))'=-1+\frac{e^{-t}}{n},\qquad  t\in ( A_n,\infty).
$$
It is clear that $\rho_n$ is bounded and continuous on $( A_n,\infty)$. Finally, the function $u_n$ satisfies four conditions \eqref{G-Cond1}-\eqref{G-Cond3}. 

Using Theorem \ref{G-Thm}, for $n\ge2$, we obtain
\begin{align*}
d_K&(\scrF_{W_n},\Psi_0)\\
&\le \tau(|1+\rho_n(W_n)|)+(1-e^{-| A_n|}) u_n( A_n+)\\
&=\int_{ A_n}^\infty |1+\rho_n(x)|u_n(x)dx+(1-e^{ A_n})\left(1-\frac{1}{n}\right) \log \left(1-\frac{1}{n} \right)^{-n} \\
&=\frac{1}{n}\int_{ A_n}^\infty e^{-2t}\exp \left(-\frac{e^{-t}}{n}\right)dt-\left\{1+ \frac{1}{n\log\left(1-\frac{1}{n}\right)} \right\} (n-1) \log \left(1-\frac{1}{n} \right) \\
&=\int_{ A_n}^\infty e^{-t} \left\{\exp \left(-\frac{e^{-t}}{n}\right)\right\}'dt-(n-1) \left\{ \log\left(1-\frac{1}{n}\right)+\frac{1}{n}\right\}\\
&=-e^{- A_n}\exp\left(-\frac{e^{- A_n}}{n}\right) + \int_{ A_n}^\infty e^{-t}\exp\left(-\frac{e^{-t}}{n}\right)dt-(n-1) \left\{ \log\left(1-\frac{1}{n}\right)+\frac{1}{n}\right\}\\
&=n \int_{ A_n}^\infty \left\{\exp\left(-\frac{e^{-t}}{n}\right)\right\}'dt-1+\frac{1}{n}\\
&=\frac{1}{n}.
\end{align*}
It is obvious that $d_K(\calF_{W_1},\Psi_0) = d_K(\calF_{X_1},\Psi_0)\le 1$. Therefore we get
$$
d_K(\calF_{W_n},\Psi_0)\le \frac{1}{n}
$$
for all $n\ge1$.
\end{example}

\section{Free Fr\'{e}chet case}

Consider $\gamma>0$. Now let us consider $Z_\gamma$ a random variable such that $\scrF_{Z_\gamma}=\calF_\gamma$ and $\{X_i\}_i$ a family of freely independent, identically distributed random variables. Suppose that  $X_i$ is distributed as a distribution function $U$ on $\R$ such that $U\in \mathcal{D}_\boxplus(\Psi_\gamma;a_n,b_n)$. Then $W_n=(X_1\lor \cdots \lor X_n-b_n)/a_n$ converges to $Z_\gamma$ in distribution as $n\rightarrow\infty$.
For $n\ge2$, we assume that $W_n$ has a density $u_n$ with respect to the Lebesgue measure on $\R$ satisfying the following conditions:
\begin{align}
&\text{supp}(u_n)=[ A_n, B_n] \text{ for some }0< A_n< B_n\le\infty, \label{F-Cond1}\\
&\lim_{t\rightarrow A_n+0}tu_n(t)\in [0,\infty) \text{ and } \lim_{t\rightarrow B_n-0}tu_n(t)=0, \label{F-Cond1-1} \\
&u_n\text{ is differentiable on } ( A_n, B_n), \label{F-Cond2}\\
&t\mapsto t\rho_n(t)  \text{ is bounded and continuous on } ( A_n, B_n). \label{F-Cond4}
\end{align}

For $\gamma>0$, we define a differential operator $\calT_\gamma$ as follows:
$$
\calT_\gamma \varphi(w):=\gamma^{-1} w\varphi'(w)-\varphi(w),
$$
for all absolutely continuous functions $\varphi$ on $(0,\infty)$.

Suppose that $x\in \R$. Consider the following ordinary differential equation
\begin{equation}\label{F-diff}
\calT_\gamma\varphi(w)=\bbm_{(-\infty,x]}(w)-\Psi_\gamma(x), \qquad w>0.
\end{equation}
The definition of $\calT_\gamma$ implies that
$$
\frac{d}{dw}\left(  w^{-\gamma} \varphi(w) \right)=\gamma w^{-\gamma-1} \left(\bbm_{(-\infty,x]}(w)-\Psi_\gamma(x)\right).
$$
Then a bounded solution of the equation \eqref{F-diff} is given by
$$
\varphi(w)=-\gamma w^{\gamma} \int_w^{\infty} \{\bbm_{(-\infty,x]}(t)-\Psi_\gamma(x)\}t^{-\gamma-1}dt, \qquad w>0,
$$
denoted by $\varphi_x$ for each $x\in \R$. Furthermore, the function $\varphi_x$ is explicitly written as follows. For $x\le 0$,
\begin{equation}\label{F-sol1}
\varphi_x(w)=0, \qquad w>0.
\end{equation}
For $0< x<1$,
\begin{equation}
\varphi_x(w)=\begin{cases}
w^\gamma x^{-\gamma}-1, &0<w\le x,\\
0, & x<w.
\end{cases}
\end{equation}
For $x\ge 1$,
\begin{equation}\label{F-sol2}
\varphi_x(w)=\begin{cases}
x^{-\gamma}(w^\gamma-1), & 0<w\le x,\\
1-x^{-\gamma}, & x<w.
\end{cases}
\end{equation}
Note that $\|\varphi_x\|_\infty \le 1$. Then we obtain $|t \varphi_x(t)u_n(t)| \le t u_n(t)\rightarrow0$ as $t\rightarrow B_n-0$ by \eqref{F-Cond1-1}. Therefore we have
\begin{align}\label{F-beta}
\lim_{t\rightarrow B_n-0}t\varphi_x(t)u_n(t)=0.
\end{align}
For each $x\in \R$ and $ A_n\ge0$, we define
$$
\eta_{\gamma, A_n,x}:=\lim_{t\rightarrow A_n+0} t\varphi_x(t)u_n(t).
$$

For $\gamma>0$ and $n\ge2$, we define the following operator:
$$
\calT_{\gamma,n}\varphi(w):=\gamma^{-1} \left(w\varphi'(w)+\varphi(w)(1+w\rho_n(w))\right), \qquad w>0,
$$
for all absolutely continuous functions $\varphi$ on $(0,\infty)$.

\begin{lem}\label{F-Lem3} 
For each $n\ge2$ and $x\in \R$, we obtain
$$
\tau(\calT_{\gamma,n}\varphi_x(W_n))= -\gamma^{-1}\eta_{\gamma, A_n,x}.
$$
Moreover, we have $|\eta_{\gamma, A_n,x}|\le  A_n  (1-( A_n\land  A_n^{-1})^{\gamma}) u_n( A_n+)<\infty$.
\end{lem}
\begin{proof}
Since $W_n$ is a self-adjoint operator affiliated with $\calM$, we have $\calT_{\gamma,n}\varphi_x(W_n)\in \calM$ by the assumption \eqref{F-Cond4}, the definition of $\calT_{\gamma,n}$ and representations \eqref{F-sol1}-\eqref{F-sol2}. From the definitions of $\calT_{\gamma,n}$ and $\eta_{\gamma, A_n,x}$ and the condition \eqref{F-beta}, we have
\begin{align*}
\tau(\calT_{\gamma,n}\varphi_x(W_n))&=\int_{ A_n}^{ B_n} \calT_{\gamma,n}\varphi_x(t) u_n(t)dt\\
&=\gamma^{-1}\int_{ A_n}^{ B_n}  \{t\varphi_x'(t)+\varphi_x(t) (1+t\rho_n(t))\} u_n(t)dt\\
&=\gamma^{-1}\int_{ A_n}^{ B_n} \{\varphi_x'(t)tu_n(t)+\varphi_x(t)( u_n(t)+tu_n'(t))\}dt\\
&=\gamma^{-1}\int_{ A_n}^{ B_n}  (\varphi_x(t)tu_n(t))'dt\\
&= -\gamma^{-1}\eta_{\gamma, A_n,x}.
\end{align*}
Thanks to \eqref{F-sol1}-\eqref{F-sol2} we observe that $|\varphi_x( A_n)|\le 1-  A_n^\gamma$ if $0<  A_n<1$ and $|\varphi_x( A_n)|\le 1-  A_n^{-\gamma}$ if $ A_n\ge 1$. Consequently, we obtain $|\varphi_x( A_n)|\le 1- ( A_n\land  A_n^{-1})^\gamma$, and therefore $|\eta_{\gamma, A_n,x}|\le  A_n \{1-( A_n\land  A_n^{-1})^\gamma\} u_n( A_n+)$. It follows from \eqref{F-Cond1-1} that $A_n \{1-( A_n\land  A_n^{-1})^\gamma\} u_n( A_n+)<\infty$.
\end{proof}

\begin{thm}\label{F-Thm}
For each $n\ge2$ and $\gamma>0$, we have 
\begin{align*}
d_K(\scrF_{W_n},\Psi_\gamma)\le  \tau(|1+\gamma^{-1} (1+W_n\rho_n(W_n))|)+ \gamma^{-1} A_n \{1-( A_n\land  A_n^{-1})^\gamma\} u_n( A_n+).
\end{align*}
\end{thm}
\begin{proof}
For each $x\in \R$, we have
\begin{align*}
\scrF_{W_n}(x)&-\Psi_\gamma(x)\\
&=\tau(\bbm_{(-\infty,x]}(W_n))-\Psi_\gamma(x)\\
&=\tau(\calT_\gamma\varphi_x(W_n)) \qquad \text{ (by definition of $\calT_\gamma$)}\\
&=\tau(\calT_\gamma\varphi_x(W_n))-\tau(\calT_{\gamma,n}\varphi_x(W_n))-\gamma^{-1}\eta_{\gamma, A_n,x} \qquad \text{ (by Lemma \ref{F-Lem3})}\\
&=-\tau\left(\varphi_x(W_n) ( 1+\gamma^{-1}(1+W_n\rho_n(W_n))) \right)-\gamma^{-1}\eta_{\gamma, A_n,x}.
\end{align*}
Therefore we obtain
\begin{align*}
|\scrF_{W_n}(x)&-\Psi_\gamma(x)| \\
&\le \tau\left(|\varphi_x(W_n)(1+\gamma^{-1}(1+W_n\rho_n(W_n)))| \right)+ \gamma^{-1}|\eta_{\gamma, A_n,x}|\\
&\le \tau \left(|1+\gamma^{-1}(1+W_n\rho_n(W_n))| \right) +\gamma^{-1}|\eta_{\gamma, A_n,x}| \qquad \text{(by $\|\varphi_x\|_\infty \le 1$)}\\
&\le \tau \left(|1+\gamma^{-1}(1+W_n\rho_n(W_n))| \right)+ \gamma^{-1} A_n \{1-( A_n\land  A_n^{-1})^\gamma\} u_n( A_n+)
\end{align*}
by Lemma \ref{F-Lem3}. By taking supremum of the left hand side for $x$, we get the inequality in this theorem.
\end{proof}

\begin{remark}
For $n\ge2$, we assume that the function $u_n$ is of the form:
$$
u_n(x)=C_n(x)x^{-\gamma-1}, \qquad x\in ( A_n, B_n),
$$
where $C_n:( A_n, B_n)\rightarrow(0,\infty)$ is differentiable on $( A_n, B_n)$ and satisfies that
\begin{align}
&\sup_{n\in \N} C_n( A_n+)<\infty, \label{F-Cond5-1}\\
& \lim_{n\rightarrow\infty} \sup_{x\in ( A_n, B_n)} \left| \frac{xC_n'(x)}{C_n(x)}\right|=0. \label{F-Cond5-2}
\end{align}
Moreover, we assume that $ A_n\rightarrow1$ as $n\rightarrow\infty$. Then we obtain
\begin{align*}
d_K&(\scrF_{W_n},\Psi_\gamma) \\
& \le \tau\left(|1+\gamma^{-1}(1+W_n\rho_n(W_n))| \right)+ \gamma^{-1} A_n \{1-( A_n\land  A_n^{-1})^\gamma\} u_n( A_n+)\\
&=\gamma^{-1}\int_{ A_n}^{ B_n}\left| \frac{xC_n'(x)}{C_n(x)}\right|u_n(x)dx+ \gamma^{-1} \{1-( A_n\land  A_n^{-1})^\gamma\}  A_n^{-\gamma} C_n( A_n+)\\
&\le \gamma^{-1}\sup_{x\in ( A_n, B_n)} \left| \frac{xC_n'(x)}{C_n(x)}\right| + \gamma^{-1}  A_n^{-\gamma} \{1-( A_n\land  A_n^{-1})^\gamma\}  \sup_{n\in \N}u_n( A_n+)\\
&\xrightarrow{n\rightarrow\infty}0,
\end{align*}
by the conditions \eqref{F-Cond5-1} and \eqref{F-Cond5-2} and that $ A_n\rightarrow1$ as $n\rightarrow\infty$.
\end{remark}

\begin{example}\label{F-ex}
Suppose that $n\ge2$. Let us consider $X_i$ a random variable distributed as the Fr\'{e}chet distribution $\calG_\gamma$ for some $\gamma>0$. Since $\Phi_\gamma\in \mathcal{D}_\boxplus(\Psi_\gamma; n^{\frac{1}{\gamma}},0)$, the random variable
$W_n=(X_1\lor \cdots \lor X_n)/n^{\frac{1}{\gamma}}$ converges to $Z_\gamma$ in distribution as  $n\rightarrow\infty$. A spectral distribution function of $W_n$ is
$$
\scrF_{W_n}(t)=\calG_\gamma^{\Box\hspace{-.55em}\lor n} (n^{\frac{1}{\gamma}}t)=\max\left\{n\exp\left(-\frac{t^{-\gamma}}{n} \right)-(n-1),0\right\}, \qquad t\in\R.
$$
A density function $u_n$ of random variable $W_n$ is given by
$$
u_n(t):=\frac{d}{dt}\scrF_{W_n}(t)=\gamma t^{-\gamma-1}\exp\left(-\frac{t^{-\gamma}}{n}\right),
$$
for all $t> A_n:=\left\{-n\log \left(1-\frac{1}{n}\right)\right\}^{-\frac{1}{\gamma}}\in (0,1)$, where $ B_n=\infty$. A straightforward computation shows that 
\begin{align*}
\lim_{t\rightarrow A_n+0} tu_n(t)&=-\gamma (n-1)\log\left(1-\frac{1}{n}\right) \in (0,\infty),\\
\lim_{t\rightarrow B_n-0} tu_n(t)&=0.
\end{align*}
It is easy to see that $u_n$ is differentiable on $( A_n,\infty)$. Moreover, we have
$$
\rho_n(t):=(\log u_n(t))'=-\frac{\gamma+1}{t}+\frac{\gamma}{n}t^{-\gamma-1} ,\qquad  t\in ( A_n,\infty).
$$
It is clear that $t\mapsto t\rho_n(t)$ is bounded and continuous on $( A_n,\infty)$. 
Therefore the function $u_n$ satisfies the conditions \eqref{F-Cond1}-\eqref{F-Cond4}.

Using Theorem \ref{F-Thm}, for $n\ge2$, we observe
\begin{align*}
d_K&(\scrF_{W_n},\Psi_\gamma)\\
&\le \tau(|1+\gamma^{-1} (1+W_n\rho_n(W_n))|)+\gamma^{-1}  A_n \{1-( A_n\land  A_n^{-1})^\gamma\} u_n( A_n+)\\
&=\int_{ A_n}^\infty |1+\gamma^{-1}(1+t\rho_n(t))|u_n(t)dt +\gamma^{-1} A_n (1- A_n^{\gamma})u_n( A_n+)\\
&=\frac{\gamma}{n}\int_{ A_n}^\infty t^{-2\gamma-1}\exp\left(-\frac{t^{-\gamma}}{n}\right)dt-(n-1) \left\{\log\left(1-\frac{1}{n}\right)+\frac{1}{n} \right\}\\
&=\int_{ A_n}^\infty t^{-\gamma} \left\{\exp\left(-\frac{t^{-\gamma}}{n}\right) \right\}'dt -(n-1)\left\{\log\left(1-\frac{1}{n}\right)+\frac{1}{n} \right\}\\
&=- A_n^{-\gamma} \exp\left(-\frac{A_n^{-\gamma}}{n} \right)  +\gamma\int_{ A_n}^\infty t^{-\gamma-1} \exp\left(-\frac{t^{-\gamma}}{n}\right)dt-(n-1)\left\{\log\left(1-\frac{1}{n}\right)+\frac{1}{n} \right\} \\
&=(n-1)\log\left(1-\frac{1}{n}\right)+ n \int_{ A_n}^\infty \left\{\exp\left(-\frac{t^{-\gamma}}{n} \right)\right\}'dt-(n-1)\left\{\log\left(1-\frac{1}{n}\right)+\frac{1}{n} \right\} \\
&=\frac{1}{n}.
\end{align*}
It is obvious that $d_K(\calF_{W_1},\Psi_\gamma)=d_K(\calF_{X_1},\Psi_\gamma)\le 1$. Finally, we obtain
$$
d_K(\calF_{W_n},\Psi_\gamma) \le \frac{1}{n}
$$
for all $n\in \N$. The rate of convergence is the same as the one in Example \ref{G-ex}. 
\end{example}

\section{Free Weibull case}

Give a number $\gamma< 0$. Let us consider $Z_\gamma$ a random variable such that $\scrF_{Z_\gamma}=\calF_\gamma$ and $\{X_i\}_i$ a family of freely independent, identically distributed random variables. Suppose that $X_i$ is distributed as a distribution function $U$ on $\R$ such that $U\in \mathcal{D}_\boxplus(\Psi_\gamma;a_n,b_n)$. Then $W_n=(X_1\lor \cdots \lor X_n-b_n)/a_n$ converges to $Z_\gamma$ in distribution as $n\rightarrow\infty$. For $n\ge2$, we assume that a density function $u_n$ of $W_n$ satisfies the conditions:
\begin{align}
&\text{supp}(u_n)=[ A_n, B_n] \text{ for some } -\infty< A_n< B_n\le 0, \label{W-Cond1}\\
&\lim_{t\rightarrow A_n+0} |t|u_n(t)\in [0,\infty) \text{ and } \lim_{t\rightarrow B_n-0}tu_n(t)=0, \label{W-Cond1-1}\\
&u_n \text{ is differentiable on } (A_n,B_n), \label{W-Cond2}\\
&t\mapsto t\rho_n(t) \text{ is bounded and continuous on } ( A_n, B_n).\label{W-Cond4}
\end{align}

For $\gamma<0$, we define a differential operator $\calT_\gamma$ as follows:
$$
\calT_\gamma \varphi(w):=-\gamma^{-1} w\varphi'(w)+\varphi(w), \qquad w<0,
$$
for all absolutely continuous functions $\varphi$ on $(-\infty,0)$. 

Suppose that $x\in \R$. Consider the following the ordinary differential equation
\begin{equation}\label{W-diff}
\calT_\gamma\varphi(w)=\bbm_{(-\infty,x]}(w)-\Psi_\gamma(x), \qquad w<0.
\end{equation}
The definition of $\calT_\gamma$ implies that
$$
\frac{d}{dw}\left\{ |w|^{-\gamma}\varphi(w)\right\}=\gamma|w|^{-\gamma-1}\{ \bbm_{(-\infty,x]}(w)-\Psi_\gamma(x)\}.
$$
Therefore a bounded solution of \eqref{W-diff} is given by
\begin{align*}
\varphi(w)=-\gamma |w|^{\gamma} \int_w^0 \{\bbm_{(-\infty,x]}(t)-\Psi_\gamma(x)\}|t|^{-\gamma-1}dt, \qquad  w<0,
\end{align*}
denoted by $\varphi_x$ for each $x\in \R$. We explicitly get a representation of the function $\varphi_x$ as follows. For $x<-1$,
\begin{equation}\label{W-sol1}
\varphi_x(w)=\begin{cases}
1-|w|^\gamma |x|^{-\gamma}, & w\le x,\\
0, & x<w<0.
\end{cases}
\end{equation}
For $-1\le x\le 0$,
\begin{equation}
\varphi_x(w)=\begin{cases}
|x|^{-\gamma}( 1-|w|^\gamma),  & w\le x,\\
|x|^{-\gamma}-1, & x<w<0.
\end{cases}
\end{equation}
For $x>0$,
\begin{equation}\label{W-sol2}
\varphi_x(w)=0, \qquad w<0.
\end{equation}
Note that $\|\varphi_x\|_\infty\le 1$. Hence we get $|t\varphi_x(t)u_n(t)|\le |t u_n(t)|\rightarrow0$ as $t\rightarrow B_n-0$ by \eqref{W-Cond1-1}. Therefore we have 
\begin{align}\label{W-beta}
\lim_{t\rightarrow B_n-0}t\varphi_x(t)u_n(t)=0.
\end{align}
For each $x\in \R$ and $ A_n\ge0$, we define
$$
\eta_{\gamma, A_n,x}:=\lim_{t\rightarrow A_n+0} t\varphi_x(t)u_n(t).
$$

For $\gamma<0$ and $n\ge2$, we define the following operator:
$$
\calT_{\gamma,n}\varphi(w):=-\gamma^{-1} \left(w\varphi'(w)+\varphi(w)(1+w\rho_n(w))\right), \qquad w<0,
$$
for all absolutely continuous functions $\varphi$ on $(-\infty,0)$. 

\begin{lem}\label{W-Lem3}
For each $n\ge2$ and $x\in \R$, we obtain
$$
\tau(\calT_{\gamma,n}\varphi_x(W_n))=\gamma^{-1}\eta_{\gamma, A_n,x}.
$$
Moreover, we have $|\eta_{\gamma, A_n,x}|\le | A_n|\{1-(| A_n|\land | A_n|^{-1})^{-\gamma}\}u_n( A_n+)<\infty$.
\end{lem}
\begin{proof}
Since $W_n$ is a self-adjoint operator affiliated with $\calM$, we notice that $\calT_{\gamma,n}\varphi_x(W_n)\in \calM$ from the assumption \eqref{W-Cond4}, the definition of $\calT_{\gamma,n}$ and representations \eqref{W-sol1}-\eqref{W-sol2}. By the definitions of $\calT_{\gamma,n}$ and $\eta_{\gamma, A_n,x}$ and the condition \eqref{W-beta}, we have
\begin{align*}
\tau(\calT_{\gamma,n}\varphi_x(W_n))&=\int_{ A_n}^{ B_n} \calT_{\gamma,n}\varphi_x(t)u_n(t)dt\\
&=-\gamma^{-1}\int_{ A_n}^{ B_n}(t\varphi_x'(t)+\varphi_x(t)(1+t\rho_n(t)))u_n(t)dt\\
&=-\gamma^{-1}\int_{ A_n}^{ B_n}(\varphi_x'(t) t u_n(t) +\varphi_x(t)(u_n(t)+tu_n'(t)))dt\\
&=-\gamma^{-1}\int_{ A_n}^{ B_n} (\varphi_x(t)tu_n(t))'dt\\
&=\gamma^{-1}\eta_{\gamma, A_n,x}.
\end{align*}
Moreover, the explicit representations \eqref{W-sol1}-\eqref{W-sol2} show that $|\varphi_x( A_n)|\le 1- | A_n|^{-\gamma}$ if $-1<  A_n < 0$ and $|\varphi_x( A_n)|\le 1- | A_n|^{\gamma}$ if $ A_n \le -1$. Consequently, we obtain $|\varphi_x( A_n)|\le 1- (| A_n|\land | A_n|^{-1})^{-\gamma}$, and therefore $|\eta_{\gamma, A_n,x}|\le | A_n| \{1-(| A_n|\land | A_n|^{-1})^{-\gamma}\} u_n( A_n+)$. By \eqref{W-Cond1-1},  the value $| A_n| \{1-(| A_n|\land | A_n|^{-1})^{-\gamma}\} u_n( A_n+)$ is finite.
\end{proof}

\begin{thm}\label{W-Thm}
For each $n\ge2$ and $\gamma<0$, we have
\begin{align*}
d_K&(\scrF_{W_n},\Psi_\gamma)\\
&\le \tau(|1+\gamma^{-1}(1+W_n\rho_n(W_n))|) +\gamma^{-1} A_n\{1-(| A_n|\land | A_n|^{-1})^{-\gamma}\}u_n( A_n+)
\end{align*}
\end{thm}
\begin{proof}
For each $x\in \R$, we have
\begin{align*}
\scrF_{W_n}(x)&-\Psi_\gamma(x)\\
&=\tau(\bbm_{(-\infty,x]}(W_n))-\Psi_\gamma(x)\\
&=\tau(\calT_\gamma\varphi_x(W_n)) \qquad \text{(by definition of $\calT_\gamma$)}\\
&=\tau(\calT_\gamma\varphi_x(W_n))-\tau(\calT_{\gamma,n}\varphi_x(W_n))+\gamma^{-1}\eta_{\gamma, A_n,x}\qquad \text{(by Lemma \ref{W-Lem3})}\\
&=\tau\left(\varphi_x(W_n)(1+\gamma^{-1}(1+W_n\rho_n(W_n)))\right) +\gamma^{-1}\eta_{\gamma, A_n,x}.
\end{align*}
Therefore we obtain
\begin{align*}
|\scrF_{W_n}(x)&-\Psi_\gamma(x)|\\
&\le \tau(|\varphi_x(W_n)(1+\gamma^{-1}(1+W_n\rho_n(W_n)))|) -\gamma^{-1} |\eta_{\gamma, A_n,x}|\\
&\le \tau(|1+\gamma^{-1}(1+W_n\rho_n(W_n))|) -\gamma^{-1} |\eta_{\gamma, A_n,x}| \qquad \text{(by $\|\varphi_x\|_\infty\le 1$)}\\
&\le \tau(|1+\gamma^{-1}(1+W_n\rho_n(W_n))|) +\gamma^{-1} A_n\{1-(| A_n|\land | A_n|^{-1})^{-\gamma}\}u_n( A_n+)
\end{align*}
by Lemma \ref{W-Lem3}. Taking the supremum of the left hand side for $x$, we get the inequality in this theorem.
\end{proof}

\begin{remark}
For $n\ge2$, we assume that the density function $u_n$ of $W_n$ is given by
$$
u_n(x)=C_n(x)|x|^{-\gamma-1}, \qquad x\in ( A_n, B_n),
$$
where $C_n:( A_n, B_n)\rightarrow (0,\infty)$ is differentiable on $( A_n, B_n)$ and satisfies that
\begin{align}
&\sup_{n\in \N} C_n( A_n+)<\infty, \label{W-Cond5-1}\\
&\lim_{n\rightarrow\infty} \sup_{x\in ( A_n, B_n)} \left|\frac{xC_n'(x)}{C_n(x)} \right|=0.\label{W-Cond5-2}
\end{align}
Moreover, we assume that $ A_n\rightarrow-1$ as $n\rightarrow\infty$. Then we have
\begin{align*}
d_K&(\scrF_{W_n},\Psi_\gamma)\\
& \le \tau \left(|1+\gamma^{-1}(1+W_n\rho_n(W_n))| \right)+ \gamma^{-1} A_n \{1-(| A_n| \land | A_n|^{-1})^{-\gamma}\} u_n( A_n+)\\
&=-\gamma^{-1}\int_{ A_n}^{ B_n}\left|\frac{xC_n'(x)}{C_n(x)} \right|u_n(x)dx- \gamma^{-1} | A_n|^{-\gamma} \{1-(| A_n|\land | A_n|^{-1})^{-\gamma}\} C_n( A_n+)\\
&\le -\gamma^{-1} \sup_{x\in ( A_n, B_n)} \left|\frac{xC_n'(x)}{C_n(x)} \right| - \gamma^{-1} | A_n|^{-\gamma} \{1-(| A_n|\land | A_n|^{-1})^{-\gamma}\} \sup_{n\in \N}C_n( A_n+)\\
&\xrightarrow{n\rightarrow\infty}0,
\end{align*}
by the conditions \eqref{W-Cond5-1} and \eqref{W-Cond5-2} and that $ A_n\rightarrow-1$ as $n\rightarrow\infty$.
\end{remark}

\begin{example}\label{W-ex}
Suppose that $n\ge2$. Let us consider $X_i$ a random variable distributed as the Weibull distribution $\calG_\gamma$ for some $\gamma<0$. Since $\Phi_\gamma \in \mathcal{D}_\boxplus(\Psi_\gamma; n^{\frac{1}{\gamma}},0)$, the random variable $W_n=(X_1\lor \cdots \lor X_n )/n^{\frac{1}{\gamma}}$ converges to $Z_\gamma$ in distribution as $n\rightarrow\infty$. A spectral distribution function of $W_n$ is given by
\begin{align*}
\scrF_{W_n}(t)=\calG_\gamma^{\Box\hspace{-.55em}\lor n} (n^{\frac{1}{\gamma}}t)=\begin{cases}
1, & t\ge0,\\
\max\left\{n\exp\left(-\frac{|t|^{-\gamma}}{n}\right)-(n-1),0\right\}, & t<0,
\end{cases}
\end{align*}
Consequently, a density function $u_n$ of random variable $W_n$ is written by
$$
u_n(t):=\frac{d}{dt}\scrF_{W_n}(t)=-\gamma|t|^{-\gamma-1}\exp\left(-\frac{|t|^{-\gamma}}{n}\right),
$$
for all $ A_n<t<0= B_n$, where $ A_n=-\left\{ -n\log\left(1-\frac{1}{n}\right)\right\}^{-\frac{1}{\gamma}} \in (-\infty, -1)$. An easy computation enables us to see that
\begin{align*}
\lim_{t\rightarrow A_n+0} |t|u_n(t)&=\gamma\left(n-1\right)\log\left(1-\frac{1}{n}\right)\in (0,\infty).\\
\lim_{t\rightarrow B_n-0} tu_n(t)&=0.
\end{align*}

It is easy to see that $u_n$ is differentiable on $( A_n,0)$. Moreover, we obtain
$$
\rho_n(t):=(\log u_n(t))'= -\frac{\gamma+1}{t}-\frac{\gamma}{n}|t|^{-\gamma-1},\qquad  t\in ( A_n,0),
$$
and therefore $t\mapsto t\rho_n(t)$ is bounded and continuous on $( A_n,0)$. Therefore the function $u_n$ satisfies the conditions \eqref{W-Cond1}-\eqref{W-Cond4}.

By Theorem \ref{W-Thm}, for $n\ge2$, we get the following estimate:
\begin{align*}
d_K&(\scrF_{W_n},\Psi_\gamma) \\
&\le\tau(|1+\gamma^{-1}(1+W_n\rho_n(W_n))|) +\gamma^{-1} A_n\{1-(| A_n|\land | A_n|^{-1})^{-\gamma}\}u_n( A_n+)\\
&=\int_{ A_n}^0 |1+\gamma^{-1}(1+t \rho_n(t))|u_n(t)dt +\gamma^{-1} A_n\{1-| A_n|^{\gamma}\}u_n( A_n+)\\
&=-\frac{\gamma}{n} \int_{ A_n}^0  |t|^{-2\gamma-1} \exp\left(-\frac{|t|^{-\gamma}}{n}\right)dt-(n-1)\log\left(1-\frac{1}{n}\right) \left\{ 1 + \frac{1}{n\log\left(1-\frac{1}{n}\right)}\right\}\\
&=\int_{ A_n}^0 |t|^{-\gamma} \left(-\frac{\gamma}{n}|t|^{-\gamma-1}\right) \exp\left(-\frac{|t|^{-\gamma}}{n}\right)dt-(n-1)\log\left(1-\frac{1}{n}\right)-1+\frac{1}{n}\\
&=\int_{ A_n}^0 |t|^{-\gamma}\left\{ \exp\left(-\frac{|t|^{-\gamma}}{n}\right)\right\}'dt-(n-1)\log\left(1-\frac{1}{n}\right)-1+\frac{1}{n}\\
&= -\gamma\int_{ A_n}^0 |t|^{-\gamma-1}\exp\left(-\frac{|t|^{-\gamma}}{n}\right)dt-1+\frac{1}{n}\\
&=n\int_{ A_n}^0\left\{ \exp\left(-\frac{|t|^{-\gamma}}{n}\right)\right\}'dt -1+\frac{1}{n}\\
&=\frac{1}{n}.
\end{align*}
It is clear that $d_K(\calF_{W_1},\Psi_\gamma)=d_K(\calF_{X_1},\Psi_\gamma)\le 1$. Hence we have
$$
d_K(\calF_{W_n},\Psi_\gamma)\le \frac{1}{n}
$$
for all $n\in \N$. The rate of convergence is the same as Examples \ref{G-ex} and \ref{F-ex} in spite of $\gamma<0$.
\end{example}

\section*{Acknowledgment}
This work was supported by JSPS KAKENHI Grant Number JP19H01791 and JSPS Open Partnership Joint Research Projects grant no. JPJSBP120209921. The author thanks to Professor Atsushi Takeuchi (Tokyo Woman's Christian University) for providing us useful information. 


\end{document}